\documentclass[12pt]{amsart}
\usepackage[utf8]{inputenc}
\usepackage{anysize}
\marginsize{3cm}{3cm}{3cm}{3cm}

\usepackage{amsmath, amssymb, amsthm}
\usepackage{tikz-cd}
\def\R{{\mathbb R}}
\def\Q{{\mathbb Q}}

\newtheorem{thm}{Theorem}
\newtheorem{lem}[thm]{Lemma}

\newtheorem{prop}[thm]{Proposition}

\title{A Finiteness Theorem in Galois Cohomology}
\author{Dylon Chow}

\begin{document}

\maketitle

\begin{abstract}
We prove the finiteness of the kernel of the localization map in the Galois cohomology of a connected reductive group over a global field.
\end{abstract}

\section{Introduction}

The purpose of this note is to prove:

\begin{thm} \label{main}
    (Borel-Serre, Borel-Prasad, Nisnevich). Let $G$ be a connected reductive group defined over a global field $F$. Let $V$ denote the set of places of $F$. The fibers of the canonical map \[\lambda \colon H^1(F,G) \to \prod_{v \in V}H^1(F_v,G)\] are finite.
\end{thm}

In more geometric language, this says that given a principal homogeneous space $X$ over $F$ for $G$, the principal homogeneous spaces over $F$ for $G$ which are isomorphic to $X$ over $F_v$ for all $v$ form a finite number of isomorphism classes over $F$. In the case when $F$ is a number field, this is ~\cite[Theorem 7.1]{BS64}. By a twisting argument, it suffices to prove that the kernel of $\lambda$ is finite. A proof that the kernel of $\lambda$ is finite is given in ~\cite[Theorem 6.8]{Borel1963}. As Platonov and Rapinchuk point out (see p.316 of ~\cite{PR}), there is a fairly standard proof in the case when $G$ is a torus, but the proof of Borel and Serre in the general case is very different and appeals to the reduction theory of adele groups.

On p.176 of ~\cite{BP90}, Borel and Prasad gave a proof in the function field case for any connected reductive group $G$. They first prove the theorem for semisimple groups in ~\cite[Theorem B.1]{BP89}. For arbitrary reductive groups, they use the finiteness of the class number of a semisimple group ~\cite[Proposition 3.9]{BP89}, the proof of which uses lower bounds on covolumes of arithmetic subgroups.

In this note, we give a totally different proof of the theorem. We take as given the case when $G$ is a torus. We prove the theorem in two steps. The first step is the case when the derived group of $G$ is simply connected. Then we use a $z$-extension to prove the theorem in general.

\section*{Acknowledgements} We would like to thank Mikhail Borovoi, Gopal Prasad and Ramin Takloo-Bighash for helpful comments.

\section{Proof of the Theorem}

First we record the statement of Theorem \ref{main} for tori:

\begin{prop} \label{tori}
    Theorem \ref{main} holds when $G$ is a torus.
\end{prop}

\begin{proof}
This follows from the global Nakayama-Tate duality theorem. See ~\cite{PR}, Theorem 6.3 and the discussion preceding it. This reference only states Nakayama-Tate duality for number fields, but the same proof works for function fields. See ~\cite{Tate1966} for a proof of Nakayama-Tate duality.
\end{proof}

In the statement below, a local field is a finite extension of either $\R, \Q_p$, or $\mathbb{F}_p((t))$ for a prime $p$.

\begin{prop} \label{wellknown}
    \begin{enumerate}
     \item If $H$ is a connected reductive group over a local field $k$, then $H^1(k,H)$ is finite.
        \item If $H$ is a semisimple simply connected group over a nonarchimedean local field $k$, then $H^1(k,H)=1$. 
        \item Let $H$ be a semisimple, simply connected algebraic group over a global field $k$. The canonical map \[H^1(k,H) \to \prod_v H^1(k_v,H)\] is injective.
    \end{enumerate}
\end{prop}

\begin{proof}
    \begin{enumerate}
        \item In characteristic $0$, this is proven in ~\cite{BS64}. For a textbook reference, see ~\cite[Theorem 6.14]{PR}. In characteristic $p>0$, see ~\cite[III.4.3, Remark 2]{Serre79}.
        \item This is due to Bruhat-Tits. See ~\cite[Thm. in section 4.7]{BT1987}.
        \item In the number field case, this was proved in ~\cite{harder1966} except for the case when $G$ has a factor of type $E_8$, which was proved in ~\cite{chernousov1989}. The function field case is proved in ~\cite{Harder1975} - in fact, in this case, $H^1(k,H)=0$.
    \end{enumerate}
\end{proof}

Now we proceed to the proof of Theorem \ref{main}. The first step is to handle the case when the derived group of $G$ is simply connected.

\begin{prop} \label{derived group}
    If the derived group of $G$ is simply connected, then Theorem 1 holds.
\end{prop}

\begin{proof}
We consider the short exact sequence \[1 \to G_{\text{der}} \to G \to D \to 1,\] where $D$ denotes the quotient of $G$ by its derived group $G_{\text{der}}$. This gives us a commutative diagram with exact rows

\begin{tikzcd}
    H^1(F,G_{\text{der}})\arrow{r}{\rho}\arrow{d}{\alpha}
    &H^1(F,G)\arrow{r}{\sigma}\arrow{d}{\beta}
    &H^1(F,D)\arrow{d}{\gamma}\\
    \prod_v H^1(F_v,G_{\text{der}})\arrow{r}{\rho'}&\prod_v H^1(F_v,G)\arrow{r}{\sigma'}&\prod_v H^1(F_v,D).
\end{tikzcd}

We want to show that $\text{ker}(\beta)$ is finite. We have \[\text{ker}(\beta) \subset \sigma^{-1}(\text{ker}(\gamma)),\] and by Proposition \ref{tori}, $\text{ker}(\gamma)$ is finite. Proposition \ref{wellknown} tells us that $H^1(F_v,G_{\text{der}})$ is trivial for all finite places $v$, that $H^1(F_v,G_{\text{der}})$ is finite for all infinite places $v$, and that $\alpha$ is injective. Together, these statements imply that $H^1(F,G_{\text{der}})$ is finite. Thus $\text{ker}(\sigma) = \text{im}(\rho)$ is also finite. Now we use a standard twisting argument (see section I.5.4 of ~\cite{Serre79}). Let $\zeta \in \text{ker}(\gamma)$, and choose a 1-cocycle $z$ that represents a class in $\sigma^{-1}(\zeta)$. By twisting the sequence above by $z$, we obtain a short exact sequence \[1 \to (G_z)_{\text{der}} \to G_z \to D \to 1.\] This induces a map $H^1(F,G_z) \to H^1(F,D)$ whose kernel is $\sigma^{-1}(\zeta)$ (after identifying $H^1(F,G)$ with $H^1(F,G_z)$). 
Therefore, for every element $\zeta \in \text{ker}(\gamma)$, $\sigma^{-1}(\zeta)$ is finite. and so $\text{ker}(\beta)$ is finite.
    
\end{proof}

We have proved Theorem \ref{main} when $G_{\text{der}}$ is simply connected. Now we use $z$-extensions to prove it in general. We choose a finite Galois extension $E/F$ splitting $G$ and an extension \[1 \to Z \to G' \to G \to 1\] such that $Z$ is an induced torus (i.e., it is obtained by Weil restriction of scalars from a split $E$-torus) and the derived group of $G'$ is simply connected. Such an extension exists - see, for instance, the discussion in the proof of Lemma A.1 of ~\cite{LM2015}. This gives a commutative diagram with exact rows

\begin{tikzcd}
    0\arrow{r}
    &H^1(F,G')\arrow{r}{\alpha}\arrow{d}{\gamma_1}
    &H^1(F,G)\arrow{r}{\beta}\arrow{d}{\gamma_2}
    &H^2(F,Z)\arrow{d}{\gamma_3}\\
    0\arrow{r}
    &\prod_v H^1(F_v,G')\arrow{r}{\alpha'}
    &\prod_v H^1(F_v,G)\arrow{r}{\beta'}
    &\prod_v H^2(F_v,Z).
\end{tikzcd}

The $0$ on the leftmost term in the top row comes from the fact that the $F$-torus $Z$ is induced. Since  $Z \times_{\text{Spec}(F)} \text{Spec}(F')$ is also induced over $F'$ for all fields $F' \supset F$, the leftmost term in the bottom row is also $0$. We will use the following lemma: 

\begin{lem} \label{global cft}
    The homomorphism $H^2(F,Z) \to \prod_v H^2(F_v,Z)$ is injective.
\end{lem}

\begin{proof}
    The torus $Z$ is the Weil restriction of some split torus over $E$, say $Z=R_{E/F}(S)$ where $S \cong \mathbb{G}_m^r$. By Shapiro's lemma we have \[H^2(F,Z) \cong H^2(E,(\mathbb{G}_{m,E})^r).\] Let $v$ be a place of $F$. Then we have (see ~\cite{milne2017}, p.58)
    \begin{align*}
        (R_{E/F}(S))_{F_v} &\cong R_{E \otimes_F F_v}(S_{E \otimes_F F_v}) \\ 
        &\cong \left(\prod_{w \lvert v} R_{E_w/F_v} (\mathbb{G}_m)\right)^r.
    \end{align*}
    Therefore, 
    \begin{align*}
        H^2(F_v,Z) &\cong \prod_{w \lvert v} H^2(F_v, R_{E_w/F_v} (\mathbb{G}_m)^r) \\
        &=\prod_{w \lvert v} H^2(E_w,\mathbb{G}_m)^r.
    \end{align*}
Thus the map $H^2(F,Z) \to \prod_v H^2(F_v,Z)$ becomes \[(\text{Br}(E))^r\to (\prod_{w \in V_E} \text{Br}(E_w))^r,\] which is injective by global class field theory.
\end{proof}

Now we return to the proof of the theorem. We want to show that $\text{ker}(\gamma_2)$ is finite. Suppose that $x \in \text{ker}(\gamma_2)$. Then $\gamma_3(\beta(x))=0$. 
By Lemma \ref{global cft}, $\beta(x)=0$. Thus $x \in \text{im}(\alpha)$, say $\alpha(w)=x$. But then $\alpha'(\gamma_1(w))=0$, and so $\gamma_1(w)=0$. Therefore, $\text{ker}(\gamma_2) \subset \alpha(\text{ker}(\gamma_1))$. Since $\text{ker}(\gamma_1)$ is finite by Proposition \ref{derived group}, $\text{ker}(\gamma_2)$ is also finite. This completes the proof of Theorem \ref{main}.

\bibliographystyle{plain}
\bibliography{biblio.bib}

\begin{thebibliography}{10}

\bibitem{Borel1963}
A.~Borel.
\newblock Some finiteness properties of adele groups over number fields.
\newblock {\em Publications Mathématiques de l'IHÉS}, 16:5--30, 1963.

\bibitem{BP89}
A.~Borel and G.~Prasad.
\newblock Finiteness theorems for discrete subgroups of bounded covolume in semi-simple groups.
\newblock {\em Publications Math{\'e}matiques de l'IH{\'E}S}, 69:119--171, 1989.

\bibitem{BP90}
A.~Borel and G.~Prasad.
\newblock Addendum : {Finiteness} theorems for discrete subgroups of bounded covolume in semi-simple groups.
\newblock {\em Publications Math\'ematiques de l'IH\'ES}, 71:173--177, 1990.

\bibitem{BS64}
A.~Borel and J.-P. Serre.
\newblock Théorèmes de finitude en cohomologie galoisienne.
\newblock {\em Commentarii mathematici Helvetici}, 39:111--164, 1964/65.

\bibitem{BT1987}
F.~Bruhat and J.~Tits.
\newblock Groupes alg{\'e}briques sur un corps local. iii. compl{\'e}ments et applications {\`a} la cohomologie galoisienne.
\newblock {\em Journal of the Faculty of Science, University of Tokyo. Section 1A, Mathematics}, 34(3):671--698, 1987.

\bibitem{chernousov1989}
V.I. Chernousov.
\newblock The hasse principle for groups of type e\_8.
\newblock In {\em Dokl. Akad. Nauk}, volume 306, pages 1059--1063, 1989.

\bibitem{harder1966}
G.~Harder.
\newblock {\"U}ber die galoiskohomologie halbeinfacher matrizengruppen. ii.
\newblock {\em Mathematische Zeitschrift}, 92(5):396--415, 1966.

\bibitem{Harder1975}
G.~Harder.
\newblock Über die galoiskohomologie halbeinfacher algebraischer gruppen. iii.
\newblock {\em Journal für die reine und angewandte Mathematik}, 0274\_0275:125--138, 1975.

\bibitem{LM2015}
E.~Lapid and Z.~Mao.
\newblock A conjecture on whittaker–fourier coefficients of cusp forms.
\newblock {\em Journal of Number Theory}, 146:448--505, 2015.
\newblock Special Issue in Honor of Steve Rallis.

\bibitem{milne2017}
J.~S. Milne.
\newblock {\em Algebraic Groups: The Theory of Group Schemes of Finite Type over a Field}.
\newblock Cambridge Studies in Advanced Mathematics. Cambridge University Press, 2017.

\bibitem{PR}
V.~Platonov and A.~Rapinchuk.
\newblock {\em Algebraic groups and number theory}.
\newblock Academic press, 1993.

\bibitem{Serre79}
J.-P. Serre.
\newblock {\em Galois cohomology}.
\newblock Springer, 1979.

\bibitem{Tate1966}
J.~Tate.
\newblock {The cohomology groups of tori in finite Galois extensions of number fields}.
\newblock {\em Nagoya Mathematical Journal}, 27(P2):709 -- 719, 1966.

\end{thebibliography}
\end{document}